\documentclass[12pt,a4paper]{article}

\usepackage[T1]{fontenc}
\usepackage{ae}	
\usepackage[italian,UKenglish]{babel}    
\usepackage{amsmath}
\usepackage{amsfonts}
\usepackage{amssymb}
\usepackage{amsthm}
\usepackage{algpseudocode}
\usepackage{float}
\usepackage{xcolor}
\usepackage{mathrsfs}
\usepackage{import}
\usepackage{geometry}

\usepackage{lettrine}   

\usepackage{verbatim}
\usepackage{alltt}

\usepackage{graphics}
\usepackage{graphicx}

\graphicspath{{./imgs/}}

\usepackage{subfigure}
\usepackage{floatflt}

\usepackage[font=small,labelfont=bf]{caption}

\theoremstyle{plain}
\newtheorem{theorem}{Theorem}[section]

\newtheorem{definition}[theorem]{Definition}

\newtheorem{proposition}[theorem]{Proposition}
\newtheorem{lemma}[theorem]{Lemma}
\newtheorem{remark}[theorem]{Remark}

\numberwithin{equation}{section}
\theoremstyle{definition}

\newcommand{\R}{\ensuremath{\mathbb{R}}}

\begin{document}

\title{Blow-up and global existence for the \\
inhomogeneous porous medium equation \\ with reaction }
%
%
\author{Giulia Meglioli\thanks{Dipartimento di Matematica, Politecnico di Milano, Italia (giulia.meglioli@polimi.it).}\,\, and Fabio Punzo\thanks{Dipartimento di Matematica, Politecnico di Milano, Italia (fabio.punzo@polimi.it).} \\ \\ \\ {\it Dedicated to the memory of Maria Assunta Pozio} }

\date{}
%

%

\maketitle              

\begin{abstract}
We study finite time blow-up and global existence of solutions to the Cauchy problem for the porous medium equation with a variable density $\rho(x)$ and a power-like reaction term. We show that for small enough initial data, if $\rho(x)\sim \frac{1}{\left(\log|x|\right)^{\alpha}|x|^{2}}$ as $|x|\to \infty$, then solutions globally exist for any $p>1$. On the other hand, when $\rho(x)\sim\frac{\left(\log|x|\right)^{\alpha}}{|x|^{2}}$ as $|x|\to \infty$, if the initial datum is small enough then one has global existence of the solution for any $p>m$, while if the initial datum is large enough then the blow-up of the solutions occurs for any $p>m$. Such results generalize those established in \cite{MP1} and \cite{MP2}, where it is supposed that $\rho(x)\sim |x|^{-q}$ for $q>0$ as $|x|\to \infty$.
\end{abstract}

\bigskip



\section{Introduction}
We are concerned with global existence and blow-up of nonnegative solutions to the Cauchy parabolic problem
\begin{equation}\label{problema}
\begin{cases}
\rho (x) u_t = \Delta(u^m)+\rho(x) u^p & \text{in }\, \R^N \times (0,\tau) \\
u=u_0& \text{in }\, \R^N \times \{0\}\,,
\end{cases}
\end{equation}
where $m>1$, $p>1$, $N\geq 3$, $\tau>0$. Furthermore, we always assume that
\begin{equation}\label{hpu0}
\begin{cases}
\textrm{(i)} \; u_0\in L^\infty(\mathbb R^N), \,\,u_0\geq 0\,\, \textrm{in}\,\, \mathbb R^N\,;\\
\textrm{(ii)} \; \rho\in C(\mathbb R^N),\, \rho>0\,\, \textrm{in}\,\, \mathbb R^N\,;
\end{cases}
\end{equation}
the function $\rho=\rho(x)$ is usually referred to as a {\it variable density}.

The differential equation in problem \eqref{problema}, posed in $(-1,1)$ with homogeneous Dirichlet boundary conditions, has been introduced in   \cite{KR3} as a mathematical model of a thermal evolution of a heated plasma. 

\smallskip

We refer the reader to \cite[Introduction]{MP1}, \cite[Introduction]{MP2} for a comprehensive account of the literature concerning various problems related to \eqref{problema}. Here we limit ourselves to recall only some contribution of that literature. 
Problem \eqref{problema} without the reaction term has been widely examined, e.g., in  \cite{Eid90,EK, GMPor, GMPfra1, GMPfra2, KKT, KPT, KP1, KP2, KRV10, KR1, KR2, PoT, PoPT2, PoPT, P1, RV06}. Furthermore, global existence and blow-up of solutions of problem \eqref{problema} with $m=1$ and $\rho\equiv 1$ have been studied, e.g., in \cite{F, H}). If
$$p\leq 1+\frac 2 N,$$ then finite time blow-up occurs, for all nontrivial nonnegative data, whereas, for
$$p>1+\frac 2 N,$$ global existence prevails for sufficiently small initial conditions. In addition, in \cite{LX} (see also \cite{dPRS}), problem \eqref{problema} with $m=1$ has been considered. 

Similar results for quasilinear parabolic equations, also involving $p$-Laplace type operators or double-nonlinear operators, have been stated in \cite{IS1}, \cite{MT}, \cite{MTS}, \cite{MTS2}, \cite{PT}, \cite{WZ} (see also \cite{GMP1} and \cite{MMP} for the case of Riemannian manifolds); moreover, in \cite{GMPhyp} the same problem on Cartan-Hadamard manifolds has been investigated.

\smallskip

Global existence and blow-up of solutions for problem \eqref{problema} with $\rho$ satisfying 
\begin{equation}\label{eq13}
\frac{1}{k_1|x|^q}\leq \rho(x)\leq \frac{1}{k_2|x|^q}\quad \textrm{for all }\,\,|x|>1
\end{equation}
have been investigated in \cite{MP1} for $q\in [0, 2)$, and in \cite{MP2} for $q\geq 2$.
In \cite{MP1}, for $q\in [0, 2)$, the following results have been established. 

\begin{itemize}
\item (\cite[Theorem 2.1]{MP1}) If $p>\overline p$, for a certain $\overline p=\overline p(k_1,k_2,q,m,N)>m$ and the initial datum is sufficiently small, then solutions exist globally in time. Observe that 
$$\overline p=m+\frac{2-q}{N-q}\,\,\,\text{when}\,\,\,k_1=k_2.$$
\item (\cite[Theorem 2.4]{MP1}) For any $p>1$, for all sufficiently large initial data, solutions blow-up in finite time.
\item (\cite[Theorem 2.6]{MP1}) For $1<p<m$, for any non trivial initial data, solutions blow-ip in finite time.
\item (\cite[Theorem 2.7]{MP1}) If  $m<p<\underline p$, for a certain $\underline p=\underline p(k_1,k_2,q,m,N)\le \overline p$, then, for any non trivial initial data, solutions blow-up in finite time, under specific extra assumptions on $\rho$.
\end{itemize}
Such results extend those stated in \cite{SGKM} for problem \eqref{problema} with $\rho\equiv 1$, $m>1$, $p>1$ (see also \cite{GV}). 

\medskip

Furthermore, assume that \eqref{eq13} holds with $q\geq 2$. In \cite{MP2} the following results have been showed.
\begin{itemize}
\item(\cite[Theorem 2.1]{MP2}) If $q=2$ and $p>m$, then, for sufficiently small initial data, solutions exist globally in time.
\item(\cite[Theorem 2.2]{MP2}) If $q=2$ and $p>m$, then, for sufficiently large initial data,  solutions blow-up in finite time.
\item(\cite[Theorem 2.3]{MP2}) If $q>2$, then, for any $p>1$, for sufficiently small initial data, solutions exist globally in time. 
\end{itemize}

%

\medskip
Finally, in \cite{GMP1}, \eqref{problema} is addressed, when $p<m$. It is assumed that \eqref{hpu0} is satisfied,
and that the weighted Poincar\'{e} inequality with weight $\rho$ holds. Moreover, in view of the assumption on $\rho$ also the weighted Sobolev inequality is fulfilled. By using such functional inequalities, it is showed 
that global existence for $L^m$ data occurs, as well as a smoothing effect for the solution, i.e. solutions corresponding to such data are bounded for any positive time. In addition, a quantitative bound on the $L^{\infty}$ norm of the solution is given.

\medskip

In what follows, we always consider two types of density functions $\rho$. To be more specific, we always make one of the following two assumptions: 
 \begin{equation}\tag{{\it $H_1$}} \label{hp1}
\begin{aligned}
&\textrm{there exist} \,\, k\in (0, +\infty)\,\, \textrm{ and }\, \alpha>1\,\, \textrm{such that}\\
&\;\;\,\, \dfrac{1}{\rho(x)}\ge k\left(\log|x|\right)^{\alpha}|x|^2\quad \textrm{for all}\,\,\, x\in \mathbb R^N\setminus B_e(0)\,;\\ 
\end{aligned}
\end{equation}
 \begin{equation}\tag{{\it $H_2$}} \label{hp2}
\begin{aligned}
&\textrm{there exist} \,\, k_1, k_2\in (0, +\infty)\,\, \textrm{with}\,\, k_1\leq k_2 \textrm{ and }\, \alpha>1\,\, \textrm{such that}\\
& \;\;\,\, k_1\dfrac{|x|^2}{\left(\log|x|\right)^{\alpha}} \le\dfrac{1}{\rho(x)}\le k_2\dfrac{|x|^2}{\left(\log|x|\right)^{\alpha}}\quad \textrm{for all}\,\,\, x\in \mathbb R^N\setminus B_e(0)\,.\\ 
\end{aligned}
\end{equation}

Assume \eqref{hp1}. For $1<p<m$ and for suitable initial data $u_0\in L^\infty(\R^N)$, we show the existence of global solutions belonging to $L^\infty(\mathbb R^N \times (0, \tau))$ for each $\tau>0.$ Indeed, in this case, the global existence follows from the results in \cite{GMP1} for $u_0\in L^m_{\rho}(\mathbb R^N).$
However, now we consider a different class of initial data $u_0$. In fact, $u_0\in L^\infty(\mathbb R^N)$ and satisfies a decaying condition as $|x|\to +\infty$; however, $u_0$ not necessarily belongs to $L^m_{\rho}(\mathbb R^N)$. 

On the other hand, for $p>m>1$, if $u_0$ satisfies a suitable decaying condition as $|x|\to +\infty$, then problem \eqref{problema} admits a solution in $L^\infty(\mathbb R^N\times (0, +\infty))$. 

Now, assume \eqref{hp2}. For any $p>m$, if $u_0$ is sufficiently large, then the solutions to problem \eqref{problema} blow-up in finite time. Moreover, 
if $p>m,$
$u_0$ has compact support and is small enough, then, under suitable assumptions on $k_1$ and $k_2$, there exist global in time solutions to problem \eqref{problema}, which belong to $L^\infty(\mathbb R^N \times (0, +\infty))$.

\smallskip

The proofs mainly relies on suitable comparison principles and properly constructed sub- and supersolutions, 
which crucially depend on the behavior at infinity of the density function $\rho(x)$. More precisely,  they are of the type
\begin{equation}\label{e4f}
w(x,t)=C\zeta(t)\left [ 1- \frac{\left(\log(|x|+r_0)\right)^{q}}{a} \eta(t) \right ]_{+}^{\frac{1}{m-1}}\quad \textrm{for any}\,\,\, (x,t)\in \big[\mathbb R^N\setminus B_e(0)\big]\times [0, T),
\end{equation}
for suitable functions $\zeta=\zeta(t), \eta=\eta(t)$ and constants $C>0, a>0, r_0>0$ and $q>1$. 
The paper is organized as follows. In Section \ref{statements} we state our main results, in Section \ref{prel} we give the precise definitions of solutions and we recall some auxiliary results. In Section \ref{gepr} we prove Theorem \ref{teo1}. The blow-up result (that is, Theorem \ref{teo3}) is proved in Section \ref{bupr}. Finally, in Section \ref{t2} Theorem \ref{teo2} is proved .

\section{Statements of the main results}\label{statements}

For any $x_0\in \mathbb R^N$ and $R>0$ we set
\begin{equation}
B_R(x_0)=\{x\in \R^N :  \| x-x_0 \| < R \}.
\end{equation}
When $x_0=0$, we write $B_R\equiv B_R(0).$

\subsection{Density $\rho$ satisfying \eqref{hp1}} 

The first result concerns the global existence of solutions to problem \eqref{problema} for any $p>1$ and $m>1$, $p\neq m$. We introduce the parameter $b\in \R$ such that
\begin{equation}\label{eq22}
0<b<\alpha-1.
\end{equation}
Moreover, since $N\geq 3$, we can choose $\varepsilon>0$ so that 
\begin{equation}\label{eq23}
N-2-\varepsilon(b+1)>0,
\end{equation}
and $r_0>e$ so that 
\begin{equation}\label{eq24}
\frac{1}{\log(|x|+r_0)}<\varepsilon\,\,\quad\text{for any}\,\,x\in\R^N.
\end{equation}
Finally, we can find $\bar c>0$ such that
\begin{equation}\label{eq25}
\left[\log(|x|+r_0)\right]^{-\frac{\bar bp}{m}}\le \bar c \quad \text{for any}\,\,x\in\R^N\,.
\end{equation}
Observe that, thanks to \eqref{hpu0}-(i) and \eqref{hp1}, we can say that there exists $k_0>0$ such that 
\begin{equation}\label{eq26}
\frac{1}{\rho(x)}\ge k_0\left[\log(|x|+r_0)\right]^{\alpha}(|x|+r_0)^2 \quad \text{for any }\,\, x\in \mathbb R^N\,.
\end{equation}

\begin{theorem}\label{teo1}
Let assumptions \eqref{hpu0}, \eqref{hp1}, \eqref{eq22}, \eqref{eq23} and \eqref{eq24} be satisfied. Suppose that $$1<p<m\,,\quad \text{or}\, \quad p> m> 1\,,$$
and that $u_0$ is small enough. Then problem \eqref{problema} admits a global solution $u\in L^\infty(\mathbb R^N\times (0, \tau))$ for any $\tau>0$.
More precisely, we have the following cases.
\begin{itemize}
\item[(a)]\, Let $1<p<m$. If $C>0$ is big enough, $T>1$, $\beta> 0$,
\begin{equation}\label{eq27}
u_0(x) \le CT^{\beta} \left (\log(|x|+r_0) \right )^{-\frac{b}{m}} \quad \text{for any}\,\, x\in \R^N\,,
\end{equation}
then problem \eqref{problema} admits a global solution $u$, which satisfies the bound from above
\begin{equation}\label{eq28}
u(x,t) \le C(T+t)^{\beta} \left (\log(|x|+r_0) \right )^{-\frac{b}{m}} \, \text{for any}\,\, (x,t)\in \R^N \times (0,+\infty)\,.
\end{equation}
\item[(b)]\, Let $p> m >  1$. If $C>0$ is small enough, $T>0$ and \eqref{eq27} holds with $\beta=0$, then problem \eqref{problema} admits a global solution $u\in L^{\infty}(\R^N\times (0,+\infty))$, which satisfies the bound from above \eqref{eq28} with $\beta=0$.
\end{itemize}
\end{theorem}

\subsection{Density $\rho$ satisfying \eqref{hp2}} 
The next result concerns the blow-up of solutions in finite time, for every $p>m>1$, provided that the initial datum is sufficiently large. We assume that hypotheses \eqref{hpu0} and \eqref{hp2} hold. In view of \eqref{hpu0}-(i), there exist $\rho_1, \rho_2\in (0, +\infty)$ with $\rho_1\leq \rho_2$ such that
\begin{equation}\label{eq215}
\rho_1\leq \frac1{\rho(x)}\leq \rho_2 \quad \textrm{for all}\,\,\, x\in \overline{B_e(0)}.
\end{equation}
Let
\begin{equation}\label{eq216}
\underline b:=\alpha+1,
\end{equation}
and
\[\mathfrak{s}(x):=\begin{cases}
\left(\log|x|\right)^{\underline b}  &\quad \text{if}\quad  x\in \R^N\setminus B_e, \\
& \\
\dfrac{\underline b \,|x|^2}{2e^2}+1-\dfrac{\underline b}{2}&\quad\text{if}\quad  x\in B_e\,.
\end{cases}\]

\begin{theorem}\label{teo3}

Let assumptions \eqref{hpu0}, \eqref{hp2}, \eqref{eq215} and \eqref{eq216} hold. For any $$p>m$$ and for any $T>0$, if the initial datum $u_0$ is large enough, then the solution $u$ of problem \eqref{problema} blows-up in a finite time $S\in (0,T]$, in the sense that
\begin{equation}\label{eq217}
\|u(t)\|_{\infty} \to \infty \text{ as } t \to S^{-}\,.
\end{equation}
More precisely, if $C>0$ and $a>0$ are large enough, $T>0$,
\begin{equation}\label{eq218}
u_0(x)\ge CT^{-\frac{1}{p-1}}\left[1-\frac{\mathfrak{s}(x)}{a}\,T^{\frac{m-p}{p-1}}\right]^{\frac{1}{m-1}}_{+} \quad \text{for any}\,\, x\in \R^N\,,
\end{equation}
then the solution $u$ of problem \eqref{problema} blows-up and satisfies the bound from below
\begin{equation} \label{eq219}
u(x,t) \ge C (T-t)^{-\frac{1}{p-1}}\left [1- \frac{\mathfrak{s}(x)}{a}\, (T-t)^{\frac{m-p}{p-1}} \right ]_{+}^{\frac{1}{m-1}}\,\, \text{for any}\,\, (x,t) \in \R^N\times(0,S)\,.
\end{equation}
\end{theorem}

Observe that if $u_0$ satisfies \eqref{eq218}, then
\begin{equation*}
\operatorname{supp}u_0\supseteq \{x\in \mathbb R^N\,:\, \mathfrak{s}(x)< a T^{\frac{p-m}{p-1}}\}\,.
\end{equation*}
From \eqref{eq219} we can infer that
\begin{equation}\label{eq220}
\operatorname{supp}u(\cdot, t)\supseteq \{x\in \mathbb R^N\,:\, \mathfrak{s}(x) < a (T-t)^{\frac{p-m}{p-1}}\}\quad \textrm{for all } t\in[0,S)\,.
\end{equation}
\medskip
The choice of the parameters $C>0, T>0$ and $a>0$ is discussed in Remark \ref{rem52}.

\medskip

The next result concerns the global existence of solutions to problem \eqref{problema} for $p>m$. We assume that $\rho$ satisfies a stronger condition than \eqref{hp2}. Indeed, we suppose that 
\begin{equation}\label{eq29}
k_1\frac{(|x|+r_0)^2}{\left(\log(|x|+r_0)\right)^{\alpha}} \le\dfrac{1}{\rho(x)}\le k_2\frac{(|x|+r_0)^2}{\left(\log(|x|+r_0)\right)^{\alpha}}\quad \textrm{for all}\,\,\, x\in \mathbb R^N\,,
\end{equation}
where 
\begin{equation}\label{eq210}
r_0 > e, \quad \frac{k_2}{k_1}<m+(N-3)\left(\frac{m-1}{\overline b}\right)\,,
\end{equation}
and
\begin{equation}\label{eq211}
\overline b:=\alpha+2.
\end{equation}

\begin{theorem}\label{teo2}
Assume \eqref{hpu0}, \eqref{eq29}, \eqref{eq210} and \eqref{eq211}. Suppose that $$p>m\,,$$
and that $u_0$ is small enough and has compact support. Then problem \eqref{problema} admits a global solution $u\in L^\infty(\mathbb R^N\times (0, +\infty))$. \newline
More precisely, if $C>0$ is small enough, $a>0$ is so that
$$
0<\omega_0\le\frac{C^{m-1}}{a}\le\omega_1
$$
for suitable $0<\omega_0<\omega_1$, $T>0$,

\begin{equation}\label{eq212}
u_0(x) \le CT^{-\frac{1}{p-1}} \left [ 1- \frac{\left(\log(|x|+r_0)\right)^{\overline b}}{a} \, T^{-\frac{p-m}{p-1}} \right ]_{+}^{\frac{1}{m-1}} \quad \text{for any}\,\, x\in \R^N\,,
\end{equation}
then problem \eqref{problema} admits a global solution $u\in L^{\infty}(\R^N\times (0,+\infty))$. Moreover,
\begin{equation}\label{eq213}
u(x,t) \le C(T+t)^{-\frac{1}{p-1}} \left [ 1- \frac{\left(\log(|x|+r_0)\right)^{\overline b}}{a} \,(T+t)^{-\frac{p-m}{p-1}} \right ]_{+}^{\frac{1}{m-1}}\end{equation}
for any $(x,t)\in \R^N \times (0,+\infty)$.
\end{theorem}

Observe that if $u_0$ satisfies \eqref{eq212}, then
\begin{equation*}
\operatorname{supp}u_0\subseteq \{x\in \mathbb R^N\,:\,\left(\log(|x|+r_0)\right)^{\overline b}\leq a T^{\frac{p-m}{p-1}}\}\,.
\end{equation*}
From \eqref{eq213} we can infer that
\begin{equation}\label{eq214}
\operatorname{supp}u(\cdot, t)\subseteq \{x\in \mathbb R^N\,:\, \left(\log(|x|+r_0)\right)^{\overline b}\leq a (T+t)^{\frac{p-m}{p-1}}\}\quad \textrm{for all } t>0\,.
\end{equation}

\medskip

The choice of the parameters $C>0, T>0$ and $a>0$ is discussed in Remark \ref{rem44}.

\bigskip

\section{Preliminaries}\label{prel}
In this section we give the precise definitions of solutions of all problems we address. Moreover, we recall some auxiliary results. The proofs can be found in \cite[Section 3]{MP1}.

\medskip

Throughout the paper we deal with {\em very weak} solutions to problem \eqref{problema} and to the same problem set in different domains, according to the following definitions.

\begin{definition}
Let $u_0\in L^{\infty}(\R^N)$ with $u_0\ge0$. Let $\tau>0$, $p>1, m>1$. We say that a nonnegative function $u\in L^{\infty}(\R^N\times (0,S))$ for any $ S<\tau$ is a solution of problem \eqref{problema} if
\begin{equation*}\label{eq31}
\begin{aligned}
-\int_{\R^N}^{}\int_{0}^{\tau} \rho(x) u \varphi_t \,dt\,dx &= \int_{\R^N} \rho(x) u_0(x) \varphi(x,0) \,dx \\ &+ \int_{\R^N}^{}\int_{0}^{\tau}  u^m \Delta \varphi \,dt\,dx \\ &+ \int_{\R^N}^{}\int_{0}^{\tau} \rho(x) u^p \varphi \,dt\,dx
\end{aligned}
\end{equation*}
for any $\varphi \in C_c^{\infty}(\R^N \times [0,\tau)), \varphi \ge 0.$ Moreover, we say that a nonnegative function $u\in L^{\infty}(\R^N\times (0,S))$ for any $ S<\tau$ is a subsolution (supersolution) if it satisfies \eqref{eq31} with the inequality $"\le"$ ($"\ge"$) instead of $"="$ with $\varphi\geq 0$.
\label{soluzioneveryweak}
\end{definition}

\begin{proposition}
Let hypotheses \eqref{hpu0} be satisfied. Then there exists a solution $u$ to problem \eqref{problema} with
$$\tau\geq \tau_0:=\frac{1}{(p-1)\|u_0\|_{\infty}^{p-1}}.$$
Moreover, $u$ is the {\em minimal solution}, in the sense that for any solution $v$ to problem \eqref{problema} there holds
\[u\leq v \quad \textrm{in }\,\,\, \mathbb R^N\times (0, \tau)\,.\]
\label{prop1}
\end{proposition}

We state the following two comparison results, which will be used in the sequel.

\begin{proposition}\label{cpsup}
Let hypothesis \eqref{hpu0} be satisfied. Let $\bar{u}$ be a supersolution to problem \eqref{problema}. Then, if $u$ is the minimal solution to problem \eqref{problema} given by Proposition \ref{prop1}, then
\begin{equation*}\label{eq34}
u\le\bar{u} \quad \text{a.e. in } \R^N \times (0,\tau)\,.
\end{equation*}
In particular, if $\bar{u}$ exists until time $\tau$, then also $u$ exists at least until time $\tau$.
\end{proposition}

\begin{proposition}\label{cpsub}
Let hypothesis \eqref{hpu0} be satisfied. Let $u$ be a solution to problem \eqref{problema} for some time $\tau=\tau_1>0$ and $\underline{u}$ a subsolution to problem \eqref{problema} for some time $\tau=\tau_2>0$. Suppose also that
$$
\operatorname{supp }\underline{u}|_{\R^N\times[0,S]} \text{ is compact for every }  \, S\in (0, \tau_2)\,.
$$
Then
\begin{equation*}\label{eq35}
u\ge\underline{u} \quad \text{ in }\,\, \R^N \times \left(0,\min\{\tau_1,\tau_2\}\right)\,.
\end{equation*}
\end{proposition}

In what follows we also consider solutions of equations of the form
\begin{equation}\label{eq36}
u_t = \frac 1{\rho(x)}\Delta(u^m) + u^p \quad \textrm{in }\,\, \Omega\times (0, \tau),
\end{equation}
where $\Omega\subseteq\mathbb R^N$ is an open subsubet. Solutions are meant in the following sense.

 \begin{definition}\label{soldom}
Let $\tau>0$, $p>1, m>1$. We say that a nonnegative function $u\in L^{\infty}(\Omega\times (0,S))$ for any $S<\tau$ is a solution of equation \eqref{eq36} if
\begin{equation}\label{eq37}
\begin{aligned}
-\int_{\Omega}^{}\int_{0}^{\tau} \rho(x) u\, \varphi_t \,dt\,dx &= \int_{\Omega}^{}\int_{0}^{\tau}  u^m \Delta \varphi \,dt\,dx \\ &+ \int_{\Omega}^{}\int_{0}^{\tau} \rho(x) u^p \varphi \,dt\,dx
\end{aligned}
\end{equation}
for any $\varphi \in C_c^{\infty}(\overline{\Omega} \times [0,\tau))$ with $\varphi| _{\partial \Omega}=0$ for all $t\in [0,\tau)$. Moreover, we say that a nonnegative function $u\in L^{\infty}(\Omega\times (0,S))$ for any $S<\tau$ is a subsolution (supersolution) if it satisfies \eqref{eq37} with the inequality $"\le"$ ($"\ge"$) instead of $"="$, with $\varphi\geq 0$.
\label{soluzioneveryweaklocale}
\end{definition}

Finally, let us recall the following well-known criterion, that will be used in the sequel.
Let $\Omega\subseteq \mathbb R^N$ be an open set. Suppose that $\Omega=\Omega_1\cup \Omega_2$ with  $\Omega_1\cap \Omega_2=\emptyset$, and that   $\Sigma:=\partial \Omega_1\cap\partial \Omega_2$ is of class $C^1$. Let $n$ be the unit outwards normal to $\Omega_1$ at $\Sigma$.
Let
\begin{equation}\label{eq38}
u=\begin{cases}
u_1 & \textrm{in }\, \Omega_1\times [0, T),\\
u_2 & \textrm{in }\, \Omega_2\times [0, T)\,,
\end{cases}
\end{equation}
where $\partial_t u\in C(\Omega_1\times (0, T)), u_1^m\in C^2(\Omega_1\times (0, T))\cap C^1(\overline{\Omega}_1\times (0, T)) , \partial_t u_2\in C(\Omega_2\times (0, T)),u_2^m\in C^2(\Omega_2\times (0, T))\cap C^1(\overline{\Omega}_2\times (0, T)).$

\begin{lemma}\label{lemext}
Let assumption \eqref{hpu0} be satisfied.

(i) Suppose that
\begin{equation}\label{eq39}
\begin{aligned}
&\partial_t u_1 \geq \frac 1{\rho}\Delta u_1^m +u_1^p \quad \textrm{for any}\,\,\, (x,t)\in \Omega_1\times (0, T),\\
&\partial_t u_2 \geq  \frac 1{\rho}\Delta u_2^m + u_2^p \quad \textrm{for any}\,\,\, (x,t)\in \Omega_2\times (0, T),
\end{aligned}
\end{equation}
\begin{equation}\label{eq310}
u_1=u_2, \quad \frac{\partial u_1^m}{\partial n}\geq \frac{\partial u_2^m}{\partial n}\quad \textrm{for any }\,\, (x,t)\in \Sigma\times (0, T)\,.
\end{equation}
Then $u$, defined in \eqref{eq38}, is a supersolution to equation \eqref{eq36}, in the sense of Definition \ref{soldom}.

(ii)  Suppose that
\begin{equation*}\label{eq185b}
\begin{aligned}
&\partial_t u_1 \leq  \frac 1{\rho}\Delta u_1^m +u_1^p \quad \textrm{for any}\,\,\, (x,t)\in \Omega_1\times (0, T),\\
&\partial_t u_2 \leq  \frac 1{\rho}\Delta u_2^m + u_2^p \quad \textrm{for any}\,\,\, (x,t)\in \Omega_2\times (0, T),
\end{aligned}
\end{equation*}
\begin{equation*}\label{eq311}
u_1=u_2, \quad\frac{\partial u_1^m}{\partial n}\leq \frac{\partial u_2^m}{\partial n}\quad \textrm{for any}\,\, (x,t)\in \Sigma\times (0, T)\,.
\end{equation*}
Then $u$, defined in \eqref{eq38}, is a subsolution to equation \eqref{eq36}, in the sense of Defi\-nition \ref{soldom}.
\end{lemma}

\section{Proof of Theorem \ref{teo1}}\label{gepr}

In what follows we set $r\equiv |x|$. We assume \eqref{hpu0}, \eqref{hp1}, \eqref{eq22} and \eqref{eq23}.  We want to construct a suitable family of supersolutions of equation
\begin{equation}\label{eq41}
u_t =\frac{1}{\rho(x)}\Delta(u^m)+u^p \quad \text{ in } \R^N\times(0,+\infty).
\end{equation}

In order to do this, we define, for all $(x,t)\in \R^N \times (0,+\infty)$,
\begin{equation}\label{eq42}
{\bar{u}}(x,t)\equiv \bar{u}(r(x),t):=C\zeta(t)\left(\log(r+r_0)\right)^{-\frac{b}{m}};
\end{equation}
where $\zeta \in C^1([0, +\infty); [0, +\infty))$, $C > 0$ and $r_0>e$ such that \eqref{eq24} is verified.

\smallskip

Now, we compute
$$
\bar{u}_t - \frac{1}{\rho}\Delta(\bar{u}^m)-\bar{u}^p.
$$

For any $(x,t) \in \big[\R^N\setminus \{0\}\big]\times (0,+\infty)$, we have:

\begin{equation}
\bar{u}_t =C\,\zeta'\, \left(\log(r+r_0)\right)^{-\frac{b}{m}}\,.
\label{eq43}
\end{equation}
\begin{equation}
(\bar{u}^m)_r=-\, b \,C^m \,\zeta^m \,\frac{\left(\log(r+r_0)\right)^{- b -1}}{r+r_0}\,.
\label{eq44}
\end{equation}
\begin{equation}
(\bar{u}^m)_{rr}= b\,C^m \,\zeta^m \left\{(b+1)\frac{\left(\log(r+r_0)\right)^{-b-2}}{(r+r_0)^2}+\frac{\left(\log(r+r_0)\right)^{-b-1}}{(r+r_0)^2}\right\}\,.
\label{eq45}
\end{equation}

\begin{proposition}\label{prop41}
Let $\zeta\in C^1([0,+\infty); [0, +\infty)), \zeta'\geq 0$. Assume \eqref{hpu0}, \eqref{hp1}, \eqref{eq22}, \eqref{eq23}, \eqref{eq24}, \eqref{eq25}, \eqref{eq26} and that
\begin{equation}\label{eq46}
k_0 b (N-2-\varepsilon( b+1))C^m\zeta^m - \bar{c} \,C^p \zeta^p\ge 0.
\end{equation}
Then $\bar u$ defined in \eqref{eq42} is a supersolution of equation \eqref{eq41}.
\end{proposition}

\begin{proof}[Proof of Proposition \ref{prop41}]
In view of \eqref{eq43}, \eqref{eq44}, \eqref{eq45}, \eqref{eq23} and \eqref{eq24}, for any $(x,t)\in (\R^N\setminus \{0\})\times (0,+\infty)$,
\begin{equation}\label{eq47}
\begin{aligned}
&\bar{u}_t-\frac 1 {\rho} \Delta(\bar{u}^m) -\bar{u}^p\\
&\geq C\zeta'\left(\log(r+r_0)\right)^{-\frac{b}{ m}} + \frac 1 {\rho} \left \{N-2-\varepsilon (b+1)\right\}C^m \zeta^m b \frac{\left(\log(r+r_0)\right)^{- b-1}}{(r+r_0)^2} \\&- C^p \zeta^p \left(\log(r+r_0)\right)^{-\frac{ bp}{m}}.
\end{aligned}
\end{equation}
Thanks to hypotheses \eqref{eq22}, \eqref{eq25} and \eqref{eq26}, we have
\begin{equation}\label{eq48}
\begin{aligned}
&\frac{1}{\rho}\frac{\left(\log(r+r_0)\right)^{-\bar b-1}}{(r+r_0)^2} \ge k_0\,\frac{\left(\log(r+r_0)\right)^{\alpha-\bar b-1}}{(r+r_0)^2} (r+r_0)^2\ge k_0\,,\\
&-\left(\log(r+r_0)\right)^{-\frac{ bp}{m}} \ge - \bar c\,.
\end{aligned}
\end{equation}
Since $\zeta'\geq 0$, from \eqref{eq48} we get
\begin{equation}\label{eq49}
\bar{u}_t-\frac 1{\rho} \Delta(\bar{u}^m) -\bar{u}^p\ge k_0\, b (N-2-\varepsilon( b+1))C^m\zeta^m - \bar{c} \,C^p \zeta^p\,.
\end{equation}
Hence \eqref{eq49} is nonnegative if
\begin{equation}\label{eq410}
k_0 \,b (N-2-\varepsilon( b+1))C^m\zeta^m - \bar{c} \,C^p \zeta^p \ge 0\,,
\end{equation}
which is guaranteed by \eqref{eq23} and \eqref{eq46}. So, we have proved that
$$
\bar{u}_t-\frac 1 {\rho} \Delta(\bar{u}^m) -\bar{u}^p \ge 0 \quad \text{in }\,\, (\R^N\setminus \{0\})\times (0, +\infty)\,.
$$
Now observe that
$$
\begin{aligned}
&\bar u \in C(\R^N\times [0,+\infty))\,,\\
&\bar u^m \in C^1([\R^N \setminus \{0\}]\times [0,+\infty))\,, \\
& \bar u_r^m(0,t)\le 0\,.
\end{aligned}
$$
Hence, thanks to a Kato-type inequality we can infer that $\bar u$ is a supersolution to equation \eqref{eq41}
in the sense of Definition \ref{soldom}.

\end{proof}

\begin{remark}\label{rem42}
Let assumption \eqref{hp1} be satisfied. In Theorem \ref{teo1} the precise hypotheses on parameters $\beta$, $C>0$, $T>0$ are as follows.
\begin{itemize}
\item[(a)] Let $p<m$. We require that
\begin{equation}\label{eq411}
\beta>0,
\end{equation}
\begin{equation}\label{eq412}
k_0 \,b (N-2-\varepsilon( b+1))C^m - \bar{c} \,C^p \ge 0\,.
\end{equation}
\item[(b)] Let $p> m$. We require that
\begin{equation}\label{eq413}
\beta=0,
\end{equation}
\begin{equation}\label{eq414}
k_0 \,b (N-2-\varepsilon( b+1))C^m - \bar{c} \,C^p\ge 0\,.
\end{equation}
\end{itemize}
\end{remark}

\begin{lemma}\label{lemma41}
All the conditions in Remark \ref{rem42} can hold simultaneously.
\end{lemma}
\begin{proof}
(a) We observe that, due to \eqref{eq23}, $$N-2-\varepsilon( b+1)>0.$$ Therefore, we can select $C>0$ sufficiently large to guarantee \eqref{eq412}.\newline
(b) We choose $C>0$ sufficiently small to guarantee \eqref{eq414}.
\end{proof}

\begin{proof}[Proof of Theorem \ref{teo1}]
We now prove Theorem \ref{teo1} in view of Proposition \ref{prop41}. In view of Lemma \ref{lemma41} we can assume that all conditions in Remark \ref{rem42} are fulfilled.
Set
$$
\zeta(t)=(T+t)^{\beta}, \quad \text{for all} \quad t \ge 0\,.
$$

Let $p<m$. Inequality \eqref{eq46} reads
\[ k_0 \,b (N-2-\varepsilon( b+1))C^m(T+t)^{m\beta}\,- \bar c \,C^p(T+t)^{p\beta} \, \ge 0\quad \textrm{for all }\,\, t>0\,.\]
This follows from \eqref{eq411} and \eqref{eq412}, for $T>1$. Hence, by Propositions \ref{prop41} and \ref{prop1} the thesis follows in this case.

Let $p>m$. Conditions \eqref{eq413} and \eqref{eq414} are equivalent to \eqref{eq46}. Hence, by Propositions \ref{prop41} and \ref{prop1} the thesis follows in this case too. The proof is complete.
\end{proof}

\section{Proof of Theorem \ref{teo3}}\label{bupr}
We construct a suitable family of subsolutions of equation
\begin{equation}\label{eq51}
u_t =\frac{1}{\rho(x)}\Delta(u^m)+u^p \quad \text{ in } \R^N\times(0, T).
\end{equation}
We assume \eqref{hpu0} and \eqref{hp2}. Let
\begin{equation}\label{eq52}
\underline w(x,t)\equiv \underline  w(r(x),t) :=
\begin{cases}
\underline  u(x,t) \quad \text{in } [\R^N \setminus B_{e}(0)] \times [0, T), \\
\underline  v(x,t) \quad \text{in } B_{e}(0) \times [0,T),
\end{cases}
\end{equation}
where
\begin{equation}\label{eq53}
\underline u(x,t)\equiv \underline u(r(x),t):=C\zeta(t)\left [1-\frac{\left(\log r\right)^{\underline b}}{a}\eta(t)\right]_{+}^{\frac{1}{m-1}}
\end{equation}
and
\begin{equation}\label{eq54}
\underline  v(x,t) \equiv \underline  v(r(x),t):=  C\zeta(t) \left [ 1-\left(\frac{\underline b r^2}{2e^2}+1-\frac{\underline b}{2}\right) \frac{\eta}{a} \right ]^{\frac{1}{m-1}}_{+}.
\end{equation}
Let 
$$
F(r,t):= 1-\frac{\left(\log r\right)^{\underline b}}{a}\eta(t)\,,
$$
and
$$
G(r,t):= 1-\left(\frac{\underline b r^2}{2e^2}+1-\frac{\underline b}{2}\right) \frac{\eta}{a}\,.
$$
Observe that  for any $(x,t) \in [\R^N\setminus B_e(0)]\times (0, T)$, we have:
\begin{equation}\label{eq55}
\underline{u}_t =C\zeta ' F^{\frac{1}{m-1}} + C\zeta \frac{1}{m-1} \frac{\eta'}{\eta}F^{\frac{1}{m-1}} - C\zeta \frac{1}{m-1} \frac{\eta'}{\eta} F^{\frac{1}{m-1}-1}.
\end{equation}
\begin{equation}\label{eq56}
(\underline  u^m)_r=-\underline b\,\frac{C^m}{a} \zeta^m \frac{m}{m-1} F^{\frac{1}{m-1}} \frac{\left(\log r\right)^{\underline b-1}}{r}\eta;
\end{equation}
\begin{equation}\label{eq57}
\begin{aligned}
(\underline{u}^m)_{rr}
&=-\underline b\,\frac{C^m}{a} \zeta^m \frac{m}{m-1} \eta\left\{F^{\frac{1}{m-1}} \left[(\underline b-1)\frac{\left(\log r\right)^{\underline b-2}}{ r^2}-\frac{\left(\log r\right)^{\underline b-1}}{r^2}\right]\right.\\
&+\frac{\underline b}{m-1}\frac{\left(\log r\right)^{\underline b-2}}{r^2}\left(1-\left(\log r\right)^{\underline b}\frac{\eta}{a}\right)F^{\frac{1}{m-1}-1}\\
&\left.-\frac{\underline b}{m-1}\frac{\left(\log r\right)^{\underline b-2}}{r^2}F^{\frac{1}{m-1}-1}\right\}\\
&=-\underline b^2\frac{C^m}{a}\left(\frac{m}{m-1}\right)^2\zeta^m \eta\frac{\left(\log r\right)^{\overline b-2}}{r^2}F^{\frac{1}{m-1}}\\
&+\underline b\frac{C^m}{a}\frac{m}{m-1}\zeta^m \eta\frac{\left(\log r\right)^{\overline b-2}}{r^2}F^{\frac{1}{m-1}}\\
&+\underline b\frac{C^m}{a}\frac{m}{m-1}\zeta^m \eta \frac{\left(\log r\right)^{\underline b-1}}{r^2}F^{\frac{1}{m-1}}\\
&+\underline b ^2\frac{C^m}{a}\frac{m}{(m-1)^2}\zeta^m \eta \frac{\left(\log r\right)^{\underline b-2}}{r^2}F^{\frac{1}{m-1}-1}\,.
\end{aligned}
\end{equation}
\begin{equation}\label{eq58}
\begin{aligned}
\Delta(\underline  u^m)&=\frac{C^m}{a} \zeta^m \eta\,\frac{m}{(m-1)^2} \,\underline b^2\,\frac{\left(\log r\right)^{\underline b-2}}{r^2}F^{\frac{1}{m-1}-1}\\
&- \frac{C^m}{a} \zeta^m \eta\,\left(\frac{m}{m-1}\right)^2 \,\underline b^2\,\frac{\left(\log r\right)^{\underline b-2}}{r^2}F^{\frac{1}{m-1}}\\
&+\frac{C^m}{a} \zeta^m \eta\,\frac{m}{m-1} \,\underline b\,\frac{\left(\log r\right)^{\underline b-2}}{r^2}F^{\frac{1}{m-1}}\\
&-\frac{C^m}{a} \zeta^m \eta\,\frac{m}{m-1} \,\underline b\,\frac{\left(\log r\right)^{\underline b-1}}{r^2}F^{\frac{1}{m-1}}(N-2)
\end{aligned}
\end{equation}

Observe that  for any $(x,t) \in B_e(0)\times (0, T)$, we have:
\begin{equation}\label{eq59}
\underline v_t =C\zeta ' G^{\frac{1}{m-1}} + C\zeta \frac{1}{m-1} \frac{\eta'}{\eta}G^{\frac{1}{m-1}} - C\zeta \frac{1}{m-1} \frac{\eta'}{\eta} G^{\frac{1}{m-1}-1},
\end{equation}
\begin{equation}\label{eq510}
(\underline  v^m)_r=-\frac{C^m}{a} \zeta^m \frac{m}{m-1} G^{\frac{1}{m-1}} \frac{\underline b r}{e^2}\eta\,,
\end{equation}
\begin{equation}\label{eq511}
\begin{aligned}
(\underline  v^m)_{rr}&=-\frac{C^m}{a} \zeta^m \frac{m}{m-1} \frac {\underline b}{e^2} \eta\left[G^{\frac{1}{m-1}} -\frac{r}{m-1} G^{\frac{1}{m-1}-1} \frac{\eta}{a} \frac{\underline b r}{e^2}\right]\,.
\end{aligned}
\end{equation}
\begin{equation}\label{eq512}
\begin{aligned}
\Delta(\underline  v^m)&=-\frac{C^m}{a} \zeta^m \frac{m}{m-1} \frac {\underline b}{e^2} \eta G^{\frac{1}{m-1}}+\frac{C^m}{a^2} \zeta^m \frac{m}{(m-1)^2} \frac {\underline b^2\, r^2}{e^4} \eta^2 G^{\frac{1}{m-1}-1}\\&-(N-1)\frac{C^m}{a} \zeta^m \frac{m}{m-1} \frac {\underline b}{e^2} \eta G^{\frac{1}{m-1}}\\
&=\frac{C^m}{a^2} \zeta^m \frac{m}{(m-1)^2} \frac {\underline b^2\, r^2}{e^4} \eta^2 G^{\frac{1}{m-1}-1}-N\,\frac{C^m}{a} \zeta^m \frac{m}{m-1} \frac {\underline b}{e^2} \eta G^{\frac{1}{m-1}}
\end{aligned}
\end{equation}

We also define
\begin{equation}\label{eq513}
\begin{aligned}
& \underline{\sigma}(t) := \zeta ' +  \frac{\zeta}{m-1} \frac{\eta'}{\eta} + \frac{C^{m-1}}{a} \zeta^m \frac{m}{m-1} \eta\, k_2 \left (\underline b\frac{m}{m-1}+N-2 \right ),\\
& \underline{\delta}(t) := \frac{\zeta}{m-1} \frac{\eta'}{\eta} \\
& \underline{\gamma}(t):=C^{p-1} \zeta^p, \\
& \underline{\sigma}_0(t) :=  \zeta ' +  \frac{\zeta}{m-1} \frac{\eta'}{\eta} + \rho_2\,N\,\frac{\underline b}{e^2}\,\frac{C^{m-1}}{a} \zeta^m \frac{m}{m-1} \eta\,,\\
& \textit{K}:=  \left (\frac{m-1}{p+m-2}\right)^{\frac{m-1}{p-1}} - \left (\frac{m-1}{p+m-2}\right)^{\frac{p+m-2}{p-1}} >0.
\end{aligned}
\end{equation}

\begin{proposition}\label{prop51}
Let $T\in (0,\infty)$, $\zeta$, $\eta \in C^1([0,T);[0, +\infty))$.
Let $\underline{\sigma},\underline{\delta},\underline{\gamma},\underline{\sigma}_0, \textit{K}$ be defined in \eqref{eq513}. Assume that, for all $t\in(0,T)$,
\begin{equation}\label{eq514}
\underline\sigma(t)>0, \quad K[\underline\sigma(t)]^{\frac{p+m-2}{p-1}} \le \underline\delta(t) \underline\gamma(t)^{\frac{m-1}{p-1}}\,,
\end{equation}
\begin{equation}\label{eq515}
(m-1)\underline \sigma(t) \le (p+m-2) \underline\gamma(t)\,.
\end{equation}
\begin{equation}\label{eq516}
\underline \sigma_0(t)>0, \quad K[\underline{\sigma}_0](t)^{\frac{p+m-2}{p-1}} \le \underline{\delta}(t) \underline{\gamma}(t)^{\frac{m-1}{p-1}}, \end{equation}
\begin{equation}\label{eq517}
(m-1) \underline{\sigma}_0(t) \le (p+m-2) \underline{\gamma}(t)\,.
\end{equation}
Then $\underline w$ defined in \eqref{eq52} is a  subsolution of equation \eqref{eq51}.
\end{proposition}

\begin{proof}[Proof of Proposition \ref{prop51}]
In view of \eqref{eq55}, \eqref{eq56}, \eqref{eq57} and \eqref{eq58}
we obtain
\begin{equation}\label{eq518}
\begin{aligned}
&\underline u_t - \frac{1}{\rho}\Delta(\underline u^m)- \underline u^p\\
&=C\zeta ' F^{\frac{1}{m-1}} + C \frac{\zeta}{m-1} \frac{\eta'}{\eta}F^{\frac{1}{m-1}} - C \frac{\zeta}{m-1} \frac{\eta'}{\eta} F^{\frac{1}{m-1}-1}\\ &- \frac{1}{\rho} \left \{ \frac{C^m}{a} \zeta^m \frac{m}{(m-1)^2} \,\underline b^2\,\eta \frac{\left(\log r\right)^{\underline b-2}}{r^2} F^{\frac{1}{m-1}-1} + \frac{C^m}{a} \zeta^m \left(\frac{m}{m-1}\right)^2 \underline b\,\eta \,\frac{\left(\log r\right)^{\underline b-2}}{r^2}F^{\frac{1}{m-1}} \right.\\&\left.- \frac{C^m}{a} \zeta^m \frac{m}{m-1} \underline b\,\eta\,\frac{\left(\log r\right)^{\underline b-2}}{r^2} F^{\frac{1}{m-1}}+ \frac{C^m}{a} \zeta^m \frac{m}{m-1} \underline b\,\eta\,\frac{\left(\log r\right)^{\underline b-1}}{r^2} F^{\frac{1}{m-1}}\,(N-2) \right \} \\ &- C^p \zeta^p F^{\frac{p}{m-1}},\quad
\textrm{for all }\, (x,t)\in D_1\,.
\end{aligned}
\end{equation}
In view of \eqref{hp2} and \eqref{eq216}, we can infer that
\begin{equation}\label{eq519}
-\frac{1}{\rho}\,\frac{\left(\log r\right)^{\underline b-2}}{r^2} \le -\frac{k_1}{\log r}\le -k_1,\quad  \textrm{for all}\,\,\, x\in \mathbb R^N\setminus B_e(0)\,,
\end{equation}
\begin{equation}\label{eq520}
\frac{1}{\rho}\,\frac{\left(\log r\right)^{\underline b-2}}{r^2} \le \frac{k_2}{\log r}\le k_2,\quad  \textrm{for all}\,\,\, x\in \mathbb R^N\setminus B_e(0)\,,
\end{equation}
\begin{equation}\label{eq521}
\frac{1}{\rho}\,\frac{\left(\log r\right)^{\underline b-1}}{r^2} \le k_2,\quad  \textrm{for all}\,\,\, x\in \mathbb R^N\setminus B_e(0)\,.
\end{equation}
From \eqref{eq518}, \eqref{eq519}, \eqref{eq520} and \eqref{eq521} we have
\begin{equation}\label{eq522}
\begin{aligned}
&\underline u_t - \frac{1}{\rho}\Delta(\underline u^m)- \underline u^p\\
& \le CF^{\frac{1}{m-1}-1} \left \{F\left [\zeta ' + \frac{\zeta}{m-1} \frac{\eta'}{\eta} + \frac{C^{m-1}}{a} \zeta^m \frac{m}{m-1} \,\underline b\,\eta k_2 \left (N-2+ \underline b \frac{m}{m-1} \right ) \right ] \right . \\
& \left .- \frac{\zeta}{m-1} \frac{\eta'}{\eta} - C^{p-1} \zeta^p F^{\frac{p+m-2}{m-1}} \right \}.
\end{aligned}
\end{equation}
Thanks to \eqref{eq513}, \eqref{eq522} becomes
\begin{equation}
\begin{aligned}
\underline u_t - \frac{1}{\rho}\Delta(\underline u^m)- \underline u^p\leq
 CF^{\frac{1}{m-1}-1} \varphi(F),
\end{aligned}
\end{equation}
where, for each $t\in (0, T)$,
$$\varphi(F):=\underline{\sigma}(t)F - \underline{\delta}(t) - \underline{\gamma}(t)F^{\frac{p+m-2}{m-1}}.$$
Our goal is to find suitable $C,a,\zeta, \eta$ such that, for each $t\in (0, T)$,
$$\varphi(F) \le 0 \quad \textrm{for any}\,\,\, F \in (0,1)\,.$$
To this aim, we impose that
\[\sup_{F\in (0,1)}\varphi(F)=\max_{F\in (0,1)}\varphi(F)= \varphi (F_0)\leq 0\,,\]
for some $F_0\in (0,1).$
We have
$$ \begin{aligned} \frac{d \varphi}{dF}=0 &\iff \underline{\sigma}(t) - \frac{p+m-2}{m-1} \underline{\gamma}(t) F^{\frac{p-1}{m-1}} =0 \\ & \iff F=F_0= \left [\frac{m-1}{p+m-2} \frac{\underline{\sigma}(t)}{\underline{\gamma}(t)} \right ]^{\frac{m-1}{p-1}} \,.\end{aligned}$$
Then
$$ \varphi(F_0)= K\, \frac{\underline{\sigma}(t)^{\frac{p+m-2}{p-1}}}{\underline{\gamma}(t)^{\frac{m-1}{p-1}}} - \underline{\delta}(t)\,, $$
where the coefficient $K$ depending on $m$ and $p$ has been defined in \eqref{eq513}. By \eqref{eq514} and \eqref{eq515}, for each $t\in (0, T)$,
\begin{equation}\label{eq524}
\varphi(F_0) \le 0\,,\quad
F_0 \le 1\,.
\end{equation}
So far, we have proved that
\begin{equation}\label{eq525}
\underline{u}_t-\frac{1}{\rho(x)}\Delta(\underline{u}^m)-\underline{u}^p \le 0 \quad \text{ in }\,\, D_1.
\end{equation}
Furthermore, since $\underline u^m\in C^1([\R^N\setminus B_e(0)]\times (0, T))$, due to Lemma \ref{lemext} (applied with $\Omega_1=D_1, \Omega_2=\R^N\setminus[B_e(0)\cup D_1], u_1=\underline u, u_2=0, u=\underline u$), it follows that $\underline u$ is a subsolution to equation
\[\underline{u}_t-\frac{1}{\rho(x)}\Delta(\underline{u}^m)-\underline{u}^p = 0 \quad \text{ in }\,\, [\mathbb R^N\setminus B_e(0)]\times (0, T),\]
in the sense of Definition \ref{soldom}.

Let
\[D_2:=\{(x,t)\in B_e(0)\times (0, T)\,:\, 0<G(r, t)<1\}\,.\]
Using \eqref{eq215}, \eqref{eq51} yields, for all $(x,t)\in D_2$,
\begin{equation}\label{eq526}
\begin{aligned}
\underline v_t- \frac{1}{\rho}&\Delta(\underline v^m) - \underline v^p  \\
&\le C G^{\frac{1}{m-1}-1}\Big\{G\left[\zeta' + \frac{\zeta}{m-1}\frac{\eta'}{\eta}+ N \,\rho_2\,\frac{\underline b}{e^2} \frac{C^{m-1}}{a}\zeta^m \frac{m}{m-1}\eta \right]\\
&-\frac{\zeta}{m-1}\frac{\eta'}{\eta} -C^{p-1}\zeta^p G^{\frac{p+m-2}{m-1}} \Big\}\\
&=C G^{\frac{1}{m-1}-1}\left[\underline{\sigma}_0(t) G - \underline{\delta}(t) - \underline{\gamma}(t)G^{\frac{p+m-2}{m-1}} \right]\,.
\end{aligned}
\end{equation}
Now, by the same arguments used to obtain \eqref{eq525}, in view of \eqref{eq517} and \eqref{eq518} we can infer that
\begin{equation}\label{eq527}
\underline  v_t - \frac 1{\rho}\Delta \underline  v^m \leq \underline  v^p\quad \textrm{for any}\,\,\, (x,t)\in D_2\,.
\end{equation}
Moreover, since $\underline  v^m\in C^1(B_e(0)\times (0, T))$, in view of Lemma \ref{lemext} (applied with $\Omega_1=D_2, \Omega_2=B_e(0)\setminus D_2, u_1=\underline v, u_2=0, u=\underline  v$), we get that
$\underline v$ is a subsolution to equation
\begin{equation}\label{eq528}
\underline v_t  - \frac 1{\rho}\Delta \underline v^m = \underline v^p\quad \textrm{in}\,\,\, B_e(0)\times (0, T)\,,
\end{equation}
in the sense of Definition \ref{soldom}. Now, observe that $\underline  w \in C(\R^N \times [0,T))$; indeed,
$$
\underline{u} = \underline  v  = C \zeta(t) \left [ 1- \frac{\eta(t)}{a} \right ]_+^{\frac{1}{m-1}} \quad \textrm{in}\,\, \partial B_e(0)\times (0, T)\,.
$$
Moreover, $\underline  w^m \in C^1(\R^N \times [0,T))$; indeed,
\begin{equation}\label{eq529}
(\underline{u}^m)_r  =  (\underline  v^m)_r  = -C^m \zeta(t)^m \frac{m}{m-1} \frac{\eta(t)}{a} \frac{\underline b}{e}\left [ 1- \frac{\eta(t)}{a} \right ]_+^{\frac{1}{m-1}} \quad \textrm{in}\,\, \partial B_e(0)\times (0, T)\,.
\end{equation}
In conclusion, in view of \eqref{eq529} and Lemma \ref{lemext} (applied with $\Omega_1=B_e(0), \Omega_2=\R^N\setminus B_e(0), u_1=\underline  v, u_2=\underline u, u=\underline  w$), we can infer that $\underline  w$ is a subsolution to equation \eqref{eq51}, in the sense of Definition \ref{soldom}.
\end{proof}

\begin{remark}\label{rem52}
Let $$p>m\,,$$ and assumptions \eqref{hp2} and \eqref{eq215} be satisfied. Let define $\omega:= \frac{C^{m-1}}{a}$. In Theorem \ref{teo3}, the precise hypotheses on parameters $C>0$, $a>0$, $\omega>0$ and $T>0$ are the following.
\begin{equation}\label{eq530}
\max \left \{1+m\,k_2\,\underline b\,\frac{C^{m-1}}{a} \left (N-2+\underline b\frac{m}{m-1}\right )\,; 1+m\rho_2\frac{C^{m-1}}{a} \,\underline b\,\frac{N}{e^2}\right \} \le (p+m-2)C^{p-1}\,,
\end{equation}
\begin{equation}\label{eq531}
\begin{aligned}
\dfrac {K}{(m-1)^{\frac{p+m-2}{p-1}}}\,\, \max& \left \{ \left [1+m\,k_2\underline b\,\frac{C^{m-1}}{a} \left (N-2+\underline b \frac{m}{m-1}\right ) \right ]^{\frac{p+m-2}{p-1}}\,\right . ;\\ &\left .\left ( 1+m\,\rho_2\,\frac{C^{m-1}}{a}\,\underline b\, \frac{N}{e^2} \right )^{\frac{p+m-2}{p-1}} \right \} \,\le \,\frac{p-m}{(m-1)(p-1)}C^{m-1} \,.
\end{aligned}
\end{equation}

\end{remark}

\begin{lemma}\label{lemma51}
All the conditions in Remark \ref{rem52} can hold simultaneously.
\end{lemma}

\begin{proof}
We can take $\omega>0$ such that
$$
\omega_0\le \omega \le \omega_1
$$
for suitable $0<\omega_0<\omega_1$ and we can choose $C>0$ sufficiently large to guarantee \eqref{eq530} and \eqref{eq531} (so, $a>0$ is fixed, too).
\end{proof}

\begin{proof}[Proof of Theorem \ref{teo3}]
We now prove Theorem \ref{teo3}, by means of Proposition \ref{prop51}. In view of Lemma \ref{lemma51} we can assume that all conditions of Remark \ref{rem52} are fulfilled. Set
$$
\zeta=(T-t)^{-\beta}\,, \quad \eta=(T-t)^{\lambda}\,, \quad \text{for all} \quad t>0\,,
$$
$$
\beta=\frac{1}{p-1}\,,\quad\quad \lambda=\frac{m-p}{p-1}\,.
$$ 
Then
\begin{equation}\label{eq532}
\begin{aligned}
& \underline\sigma(t) := \left [ \frac{1}{m-1}+\frac{C^{m-1}}{a}\frac{m}{m-1}\underline b\,k_2\,\left(\underline b\,\frac m {m-1}+N-2\right)\right] (T-t)^{-\frac{p}{p-1}}\,,\\
&\underline \delta(t) := \frac{p-m}{(m-1)(p-1)} (T-t)^{-\frac{p}{p-1}}\,,  \\
& \underline\gamma(t):=C^{p-1} (T-t)^{-\frac{p}{p-1}}\,, \\
& \underline \sigma_0(t) :=  \frac 1 {m-1} \left [1+\frac{\rho_2\, N\, m\,\underline b}{e^2} \,\frac{C^{m-1}}{a} \right](T-t)^{-\frac{p}{p-1}} \,.
\end{aligned}
\end{equation}
Let $p>m$. Condition \eqref{eq530} implies \eqref{eq515}, \eqref{eq517}, while condition \eqref{eq531} implies \eqref{eq514}, \eqref{eq516}. Hence by Propositions \ref{prop51} and \ref{cpsub} the thesis follows.
\end{proof}

\section{Proof of Theorem \ref{teo2}}\label{t2}  
We assume \eqref{hpu0}, \eqref{eq29} and \eqref{eq210}. In order to construct a suitable family of supersolutions of \eqref{eq41}, we define, for all $(x,t)\in \R^N \times (0,+\infty)$,
\begin{equation}\label{eq415}
{\bar{u}}(x,t)\equiv \bar{u}(r(x),t):=C\zeta(t)\left [1-\frac{\left(\log(r+r_0)\right)^{\overline b}}{a}\eta(t)\right]_{+}^{\frac{1}{m-1}}\,,
\end{equation}
where $\eta$, $\zeta \in C^1([0, +\infty); [0, +\infty))$, $C > 0$, $a > 0$, $r_0>e$ and $\overline b$ as in \eqref{eq211}.

\smallskip

Now, we compute
$$
\bar{u}_t - \frac{1}{\rho}\Delta(\bar{u}^m)-\bar{u}^p.
$$
To this aim, set
$$
F(r,t):= 1-\frac{\left(\log(r+r_0)\right)^{\overline b}}{a}\,\eta(t)\,,
$$
and 
$$
\textit{D}_1:=\left \{ (x,t) \in [\R^N \setminus \{0\}]  \times (0,+\infty)\,\, | \,\,0<F(r,t)<1 \right \}.
$$
For any $(x,t) \in \textit{D}_1$, we have:

\begin{equation}\label{eq416}
\begin{aligned}
\bar{u}_t &=C\zeta ' F^{\frac{1}{m-1}} + C\zeta \frac{1}{m-1} F^{\frac{1}{m-1}-1} \left ( -\frac{\left(\log(r+r_0)\right)^{\overline b}}{a} \eta ' \right )\\
&=C\zeta ' F^{\frac{1}{m-1}} + C\zeta \frac{1}{m-1} \frac{\eta'}{\eta}F^{\frac{1}{m-1}} - C\zeta \frac{1}{m-1} \frac{\eta'}{\eta} F^{\frac{1}{m-1}-1}. \\
\end{aligned}
\end{equation}
\begin{equation}\label{eq417}
(\bar{u}^m)_r=-\overline b\,\frac{C^m}{a} \zeta^m \frac{m}{m-1} F^{\frac{1}{m-1}} \frac{\left(\log(r+r_0)\right)^{\overline b-1}}{(r+r_0)}\eta.
\end{equation}
\begin{equation}\label{eq418}
\begin{aligned}
(\bar{u}^m)_{rr}&=-\overline b\,\frac{C^m}{a} \zeta^m \frac{m}{m-1} \eta\left\{F^{\frac{1}{m-1}} \left[(\overline b-1)\frac{\left(\log(r+r_0)\right)^{\overline b-2}}{(r+r_0)^2}-\frac{\left(\log(r+r_0)\right)^{\overline b-1}}{(r+r_0)^2}\right]\right.\\
&+\frac{\overline b}{m-1}\frac{\left(\log(r+r_0)\right)^{\overline b-2}}{(r+r_0)^2}\left(1-\left(\log(r+r_0)\right)^{\overline b}\frac{\eta}{a}\right)F^{\frac{1}{m-1}-1}\\
&\left.-\frac{\overline b}{m-1}\frac{\left(\log(r+r_0)\right)^{\overline b-2}}{(r+r_0)^2}F^{\frac{1}{m-1}-1}\right\}\\
&=-\overline b\frac{C^m}{a}\frac{m}{m-1}\zeta^m \eta \left[\overline b\frac{m}{m-1}-1\right]\frac{\left(\log(r+r_0)\right)^{\overline b-2}}{(r+r_0)^2}F^{\frac{1}{m-1}}\\
&+\overline b\frac{C^m}{a}\frac{m}{m-1}\zeta^m \eta \frac{\left(\log(r+r_0)\right)^{\overline b-1}}{(r+r_0)^2}F^{\frac{1}{m-1}}\\
&+\overline b ^2\frac{C^m}{a}\frac{m}{(m-1)^2}\zeta^m \eta \frac{\left(\log(r+r_0)\right)^{\overline b-2}}{(r+r_0)^2}F^{\frac{1}{m-1}-1}\,.
\end{aligned}
\end{equation}
\begin{equation}\label{eq419}
\begin{aligned}
\Delta(\bar{u}^m)
&= \frac{(N-1)}{r}(\bar{u}^m)_r + (\bar{u}^m)_{rr} \\
&= \frac{(N-1)}{r}\left (-\overline b\,\frac{C^m}{a} \zeta^m \frac{m}{m-1} F^{\frac{1}{m-1}} \frac{\left(\log(r+r_0)\right)^{\overline b-1}}{(r+r_0)}\eta \right ) \\
&-\overline b\frac{C^m}{a}\frac{m}{m-1}\zeta^m \eta \left[\overline b\frac{m}{m-1}-1\right]\frac{\left(\log(r+r_0)\right)^{\overline b-2}}{(r+r_0)^2}F^{\frac{1}{m-1}}\\
&+\overline b\frac{C^m}{a}\frac{m}{m-1}\zeta^m \eta \frac{\left(\log(r+r_0)\right)^{\overline b-1}}{(r+r_0)^2}F^{\frac{1}{m-1}}\\&+\overline b ^2\frac{C^m}{a}\frac{m}{(m-1)^2}\zeta^m \eta \frac{\left(\log(r+r_0)\right)^{\overline b-2}}{(r+r_0)^2}F^{\frac{1}{m-1}-1}\,.
\end{aligned}
\end{equation}

We also define
\begin{equation}\label{eq420}
\begin{aligned}
&\bar{\sigma}(t) := \zeta ' + \frac{\zeta}{m-1} \frac{\eta'}{\eta} + \overline b \frac{C^{m-1}}{a} \zeta^m \frac{m}{m-1} \eta k_1 \left (\overline b \frac{m}{m-1}+N-3\right ),\\
&\bar{\delta}(t) :=  \frac{\zeta}{m-1} \frac{\eta'}{\eta} + \overline b^2\,\frac{C^{m-1}}{a} \zeta^m \frac{m}{(m-1)^2}\eta k_2\,, \\
& \bar{\gamma}(t):=C^{p-1}\zeta^p\,.
\end{aligned}
\end{equation}

\begin{proposition}\label{prop43}
Let $\zeta, \eta\in C^1([0,+\infty);[0, +\infty))$. Let $\bar\sigma$, $\bar\delta$, $\bar\gamma$ be as defined in \eqref{eq420}. Assume \eqref{hp2}, \eqref{eq29}, \eqref{eq210}, \eqref{eq211} and that, for all $t\in (0,+\infty)$,
\begin{equation}\label{eq421}
-\frac{\eta'}{\eta^2} \ge \overline b^2\,\frac{C^{m-1}}{a} \zeta^{m-1}\frac{m}{m-1}k_2,
\end{equation}
and
\begin{equation}\label{eq422}
\zeta' + \overline b \,\frac{C^{m-1}}{a} \zeta^m \frac{m}{m-1} \eta \left[\left(\overline b\frac{m}{m-1}+N-3\right)k_1-\frac{\overline b}{(m-1)}k_2\right] -C^{p-1}\zeta^p \ge 0\,.
\end{equation}
Then $\bar{u}$ defined in \eqref{eq415} is a supersolution of equation \eqref{eq41}.
\end{proposition}

\begin{proof}[Proof of Proposition \ref{prop43}]

In view of \eqref{eq416}, \eqref{eq417}, \eqref{eq418} and \eqref{eq419}, for any $(x,t)\in D_1$,
\begin{equation}\label{eq423}
\begin{aligned}
\bar{u}_t - &\frac{1}{\rho}\Delta(\bar{u}^m)-\bar{u}^p \\
\ge\,\,
&C\zeta ' F^{\frac{1}{m-1}} + C\zeta \frac{1}{m-1} \frac{\eta'}{\eta}F^{\frac{1}{m-1}} - C\zeta \frac{1}{m-1} \frac{\eta'}{\eta} F^{\frac{1}{m-1}-1}\\ &+\frac{1}{\rho}(N-2)\,\overline b\,\frac{C^m}{a} \zeta^m \frac{m}{m-1} F^{\frac{1}{m-1}} \frac{\left(\log(r+r_0)\right)^{\overline b-1}}{(r+r_0)^2}\eta \\
&+\frac{1}{\rho}\,\overline b\,\frac{C^m}{a}\frac{m}{m-1}\zeta^m \eta \left[\overline b\frac{m}{m-1}-1\right]\frac{\left(\log(r+r_0)\right)^{\overline b-2}}{(r+r_0)^2}F^{\frac{1}{m-1}}\\
&-\frac{1}{\rho}\,\overline b ^2\frac{C^m}{a}\frac{m}{(m-1)^2}\zeta^m \eta \frac{\left(\log(r+r_0)\right)^{\overline b-2}}{(r+r_0)^2}F^{\frac{1}{m-1}-1}- C^p \zeta^p F^{\frac{p}{m-1}},
\end{aligned}
\end{equation}
where we have used the inequality
$$
\frac{1}{r(r+r_0)}\ge\frac{1}{(r+r_0)^2}\,.
$$
Thanks to  \eqref{eq29} and \eqref{eq211}, we have
\begin{equation}\label{eq424}
\frac{1}{\rho}\,\frac{\left(\log(r+r_0)\right)^{\overline b-2}}{(r+r_0)^2} \ge k_1\quad \text{for all }\,\, x\in \R^N\,,
\end{equation}
\begin{equation}\label{eq425}
-\frac{1}{\rho}\,\frac{\left(\log(r+r_0)\right)^{\overline b-2}}{(r+r_0)^2} \ge -k_2 \quad \text{for all }\,\, x\in \R^N\,,
\end{equation}
\begin{equation}\label{eq426}
\frac{1}{\rho}\,\frac{\left(\log(r+r_0)\right)^{\overline b-1}}{(r+r_0)^2} \ge k_1\log(r+r_0)\ge k_1 \quad \text{for all }\,\, x\in \R^N\,.
\end{equation}
From \eqref{eq424}, \eqref{eq425} and \eqref{eq426} we get
\begin{equation}\label{eq427}
\begin{aligned}
&\bar{u}_t - \frac{1}{\rho}\Delta(\bar{u}^m)-\bar{u}^p  \\
&\ge CF^{\frac{1}{m-1}-1} \left \{F\left [\zeta ' + \frac{\zeta}{m-1} \frac{\eta'}{\eta} + \overline b \,\frac{C^{m-1}}{a} \zeta^m \frac{m}{m-1} \eta\, k_1 \left (\overline b \frac{m}{m-1}+N-3\right ) \right ] \right . \\
& \left .-\frac{\zeta}{m-1} \frac{\eta'}{\eta} - \overline b^2\,\frac{C^{m-1}}{a} \zeta^m \frac{m}{(m-1)^2}\eta\, k_2 - C^{p-1} \zeta^p F^{\frac{p+m-2}{m-1}} \right \}
\end{aligned}
\end{equation}
From \eqref{eq427} and \eqref{eq420}, we have
\begin{equation}\label{eq6}
\bar u_t - \frac{1}{\rho}\Delta(\bar u^m)-\bar u^p\geq C F^{\frac{1}{m-1}-1} \left[\bar{\sigma}(t)F - \bar{\delta}(t) - \bar{\gamma}(t)F^{\frac{p+m-2}{m-1}}\right ]\,.
\end{equation}
For each $t>0$, set
$$\varphi(F):=\bar{\sigma}(t)F - \bar{\delta}(t) - \bar{\gamma}(t)F^{\frac{p+m-2}{m-1}}, \quad F\in (0,1)\,.$$
Now our goal is to find suitable $C,a,\zeta, \eta$ such that, for each $t>0$,
$$\varphi(F) \ge 0 \quad \text{for any }\, F\in (0,1)\,.$$
We observe that $\varphi(F)$ is concave in the variable $F$. Hence it is sufficient to have that $\varphi(F)$ is positive at the extrema of the interval $(0,1)$. This reduces, for any $t>0$, to the conditions
\begin{equation}
\begin{aligned}
&\varphi(0) \ge 0\,,  \\
&\varphi(1) \ge 0  \,.
\end{aligned}
\end{equation}
These are equivalent to
$$
-\bar\delta(t) \ge 0\,, \quad \bar\sigma(t)-\bar\delta(t)-\bar\gamma(t) \ge 0\,,
$$
that is
$$
\begin{aligned}
&-\frac{\eta'}{\eta^2} \ge \overline b^2\,\frac{C^{m-1}}{a} \zeta^{m-1}\frac{m}{m-1}k_2\,,\\
&\zeta' + \overline b \,\frac{C^{m-1}}{a} \zeta^m \frac{m}{m-1} \eta \left[\left(\overline b\frac{m}{m-1}+N-3\right)k_1-\frac{\overline b}{(m-1)}k_2\right] -C^{p-1}\zeta^p \ge 0\,.
\end{aligned}
$$
which are guaranteed by \eqref{eq210}, \eqref{eq421} and \eqref{eq422}. Hence we have proved that
$$
\bar{u}_t - \frac{1}{\rho}\Delta(\bar{u}^m)-\bar{u}^p \ge 0 \quad \text{in } \textit{D}_1\,.
$$
Now observe that
\begin{itemize}
\item[]$\bar{u} \in C(\R^N \times [0,+\infty))$\,,
\item[]$\bar{u}^m \in C^1([\R^N \setminus \{0\}] \times [0,+\infty))$\,, and by the definition of $\bar{u}$\,,
\item[]$\bar{u}\equiv0$ in $[\R^N\setminus \textit{D}_1] \times [0,+\infty))$\,.
\end{itemize}
Hence, by Lemma \ref{lemext} (applied with $\Omega_1=D_1$, $\Omega_2=\R^N\setminus D_1$, $u_1=\bar u$, $u_2=0$, $u=\bar u$), $\bar u$ is a supersolution of equation
$$
\bar{u}_t - \frac{1}{\rho}\Delta(\bar{u}^m)-\bar{u}^p = 0 \quad \text{in } (\R^N \setminus \{0\})\times (0,+\infty)
$$
in the sense of Definition \ref{soldom}. Thanks to a Kato-type inequality,  since $\bar u^m_r(0,t)\le 0$, we can easily infer that $\bar u$ is a supersolution of equation \eqref{eq41} in the sense of Definition \ref{soldom}.
\end{proof}

\begin{remark}\label{rem44}
Let $$p>m$$ and assumption \eqref{eq210} be satisfied. Let $\omega:=\frac{C^{m-1}}{a}$. In Theorem \ref{teo2} the precise hypotheses on parameters $C>0$, $\omega>0$, $T>0$ are the following:
\begin{equation}\label{eq430}
\frac{p-m}{p-1} \ge \overline b^2\,\omega \frac{m}{m-1} k_2,
\end{equation}
\begin{equation}\label{eq431}
\overline b\,\omega \frac{m}{m-1}\left [ k_1\left(\overline b \frac{m}{m-1}+N-3\right)-\frac{k_2}{(m-1)}\,\overline b\right ] \ge C^{p-1} + \frac{1}{p-1}\,.
\end{equation}
\end{remark}

\begin{lemma}\label{lemma42}
All the conditions in Remark \ref{rem44} can be satisfied simultaneously.
\end{lemma}
\begin{proof}
Since $p>m$ the left-hand-side of \eqref{eq430} is positive. By \eqref{eq210}, we can select $\omega>0$ so that \eqref{eq430} holds and
$$
\overline b\,\omega \frac{m}{m-1}\left [ k_1\left(\overline b \frac{m}{m-1}+N-3\right)-\frac{k_2}{(m-1)}\,\overline b\right ] \ge \frac{1}{p-1}\,.
$$
Then we take $C>0$ so small that \eqref{eq431} holds (and so $a>0$ is accordingly fixed).
\end{proof}

\begin{proof}[Proof of Theorem \ref{teo2}]
In view of Lemma \ref{lemma42}, we can assume that all the conditions in Remark \ref{rem44} are fulfilled.
Set
$$
\zeta(t)=(T+t)^{-\frac{1}{p-1}}, \quad \text{for all} \quad t \ge 0\,,
$$
and
$$
\eta(t)=(T+t)^{-\frac{p-m}{p-1}}, \quad \text{for all} \quad t \ge 0\,.
$$
Let $p>m$. Consider conditions \eqref{eq421} and \eqref{eq422} with this choice of $\zeta$ and $\eta$. They read
$$
\frac{p-m}{p-1}\, \ge \bar b^2 \frac{C^{m-1}}{a}\frac{m}{m-1}k_2,
$$
$$
-\frac{1}{p-1} + \overline b \,\frac{C^{m-1}}{a}\frac{m}{m-1} \left[\left(\overline b\frac{m}{m-1}+N-3\right)k_1-\frac{\overline b}{(m-1)}k_2\right] -C^{p-1}\ge 0\,.
$$
Therefore, \eqref{eq421} and \eqref{eq422} follow from assumptions \eqref{eq430} and \eqref{eq431}. Hence, by Propositions \ref{prop43} and \ref{prop1} the thesis follows.
\end{proof}

%
%

\end{document}